\documentclass[12pt]{article}
\usepackage{amssymb,amsmath}
\usepackage{amsthm}
\usepackage{amscd}
\usepackage{stmaryrd}
\newtheorem{theorem}{Theorem}[section]
\newtheorem{theorem*}{Theorem A\!\!}
\newtheorem{proposition}{Proposition}[section]
\newtheorem{proposition*}{Proposition A\!\!}
\newtheorem{corollary}{Corollary}[section]
\newtheorem{corollary*}{Corollary A\!\!}

\newtheorem{lemma}{Lemma}[section]

\newtheorem{definition}{Definition}[section]

\DeclareMathOperator{\End}{End}

\DeclareMathOperator{\Mat}{Mat}

\DeclareMathOperator{\Aut}{Aut}
\DeclareMathOperator{\rank}{rank}
\DeclareMathOperator{\id}{id}
\DeclareMathOperator{\Id}{Id}

\DeclareMathOperator{\tr}{tr}

\DeclareMathOperator{\im}{Im} 
    
\DeclareMathOperator{\Lie}{Lie}

\DeclareMathOperator{\Sym}{Sym}
\DeclareMathOperator{\Str}{Str}
\DeclareMathOperator{\Det}{Det}
\DeclareMathOperator{\proj}{proj}

\DeclareMathOperator{\res}{res}

\DeclareMathOperator{\Der}{Der}

\DeclareMathOperator{\Symm}{Symm}

\begin{document}
\title{Construction \`a la Ibukiyama\\ of symmetry breaking differential operators, I \\ }

\author{Jean-Louis Clerc}

\date{}
\maketitle
\abstract{The construction of symmetry breaking differential operators,  using invariant pluri-harmonic polynomials, due to T. Ibukiyama in the context of the Siegel upper half space, is extended for scalar representations to general Hermitian symmetric spaces of tube-type. The new context is described in terms of Euclidean Jordan  algebras and their representations. As an example, new and explicit differential operators  are obtained for the restriction from the tube domain over the light cone to the product of two upper half-planes.}
\bigskip

{2020 MSC. Primary 32M15; Secondary 17C20, 22E46, 11F70}

\section*{Introduction}

In his seminal paper \cite{i} T. Ibukiyama introduced a construction of holomorphic differential operators, in the geometric context of the Siegel upper half space, its group of holomorphic diffeomorphisms (the symplectic group) and the holomorphic series of representations. The differential operators he constructs are examples of symmetry breaking differential operators in the sense of T. Kobayashi (see \cite{k}). They can also  be viewed as generalizations of the classical Rankin-Cohen brackets, and they play an important r\^ole in the theory of Siegel modular  forms.  For more on the subject, see \cite{i2} and the bibliography therein.

 In the present paper, a broader geometric context is considered, namely  Hermitian symmetric spaces of tube-type. The theory of  Euclidean Jordan algebras is very useful to handle these situation and  the (lesser known) notion of \emph{representation} of a Euclidean Jordan algebra, as introduced in \cite{fk} and further studied in \cite{c92,c00, c02} is central to the present article. 
 
 The main theorem of \cite{i} is rephrased and proved in this broader context.  However, we only explore the situation called (I) in \cite{i}, and we  only consider  the case of \emph{scalar} holomorphic representations.
 
As a test of the efficiency of the process to produce new and explicit  examples, some symmetry breaking differential operators are obtained, in relation with  the restriction from the tube-domain over the light cone to the product of two upper-half planes (or equivalently from  the Lie ball to the bi-disc). 
 \bigskip
 
 \centerline{\bf CONTENTS}
 \medskip
 
 {\bf 1. The algebraic/geometric setting}
\medskip

\hskip 1cm {\bf 1.1} Compete system of orthogonal idempotents and the 

\hskip 1.7cm associated Jordan subalgebra
\smallskip

\hskip 1cm {\bf 1.2} The group $L(\mathbf c)$
\smallskip

\hskip 1cm {\bf 1.3} The group $G(\mathbf c)$
\smallskip

\hskip 1cm {\bf 1.4} $\mathbf c$-plurihomogeneous polynomials and $L(\mathbf c)$-covariant

\hskip 1.7cm differential operators

{\bf 2. Jordan algebra representations}
\medskip

\hskip 1cm {\bf 2.1} Generalities
\smallskip

\hskip 1cm {\bf 2.2} Restriction to $J(\mathbf c)$
\smallskip

{\bf 3. Pluri-harmonic polynomials}
\medskip

\hskip 1cm {\bf 3.1} The Hecke formula for puri-harmonic polynomials
\smallskip

\hskip 1cm {\bf 3.2} $\mathbf c$-pluri-harmonic polynomials
\medskip

{\bf 4. Holomorphic representations}
\medskip

{\bf 5. The main theorem}
\medskip

\hskip 1cm {\bf 5.1} The data and the statement of the main theorem
\smallskip

\hskip 1cm {\bf 5.2} The proof of the main theorem
\medskip

{\bf 6. Examples in rank 2}
\medskip

To complete the introduction, here is a more precise description of this paper. Let $D$ be a Hermitian symmetric space, and let $\Aut(D)$ be its group of holomorphic diffeomorphisms. Let $D'$  be a  Hermitian symmetric subspace of $D$. Let $G$ be the subgroup of $\Aut(D)$ which preserves the smaller domain $D'$. Let $\pi$ a representation of $\Aut(D)$ realized on space $\mathcal O(D)$ of holomorphic functions on $D$, and let $\pi'$ be a representation of $G$ which is realized on the space $\mathcal O(D')$. Finally let $\res: \mathcal O(D) \longrightarrow\mathcal O(D')$ be the restriction map. In this context, a \emph{symmetry breaking differential operator} (SBDO for short) is a holomorphic differential operator $\mathcal D$ on $D$ such that, for any $g\in G$
\[(\res\circ\mathcal D) \circ \pi(g) = \pi'(g) (\res \circ \mathcal D)\ .
\]

Among the Hermitian symmetric domains, there is the subclass of domains of tube-type, those which can be realized as  Siegel domains of type I, i.e. complex tubes over convex symmetric cones in a Euclidean space. In turn, symmetric cones are related to Euclidean Jordan algebras. More precisely to any Euclidean Jordan algebra $J$ is associated a convex symmetric cone $\Omega$ and a Hermitian symmetric space of tube-type $T_\Omega$, and vice versa, in a very functorial way (see \cite{s} Ch. I Section 9). The notion of \emph{complete system of orthogonal idempotents} allows to construct  specific Jordan subalgebras $J'$ such that the associated tube-type domain $T_{\Omega'}$ is a Hermitian subdomain of $T_\Omega$, both of the same rank. An example is the  situation studied by Ibukiyama, where $J= \Symm(r)$ is the Jordan algebra of real symmetric matrices of size $r$,
 $J'= \Symm(r_1)\oplus \dots\oplus \Symm(r_k)$ and $r_1+r_2=\dots +r_k=r$. The corresponding tube-type domains are the Siegel upper half-space $T_\Omega=\mathbb H_r$ and $T_{\Omega'} = \mathbb H_{r_1} \oplus \dots \oplus \mathbb H_{r_k}$. 

Section 1 is devoted to study this situation for a general Euclidean Jordan algebra and a general CSOI. The structure of the subgroup of $\Aut(T_\Omega)$ which preserve the smaller tube-type domain $T_{\Omega'}$ is precisely described. 

The notion of \emph{representation of a Euclidean Jordan algebra}, systematically introduced by J. Faraut and A. Kor\'anyi (see \cite{fk}) offers a nice framework to reinterpret and generalize Ibukiyama's construction. The notion is recalled in Section 2, and further developped in the context of Section 1.  The notion of pluri-harmonic polynomials is introduced in Section 3, and  the classical Hecke formula  (already extended in \cite {kv}) is further extended to the present situation.

Section 4 recalls the construction of the (scalar) holomorphic series of representations for the group $\Aut(T_\Omega)$, in fact for a twofold covering of its neutral component.

Section 5 contains the main result, and reduces the analytic problem to an algebraic problem about a class of polynomials on $J$. The problem is hard to solve in general, but many cases can be investigated, using in particular the classical theory of compact simple Lie groups and theory of invariants (see \cite {i2}).

Section 6 is devoted to an example, corresponding to the case where the Jordan algebra $J$ is of rank 2. The domain $T_\Omega$ is the tube-domain over the light cone, the subdomain $T_{\Omega'}$ is a product of two upper half-planes. The representations of $J$ involved are interpreted as \emph{Clifford modules}. In this case, we investigate the algebraic problem and give explicit solutions, producing new pluri-harmonic polynomials and new SBDO.
 
\section{The algebraic/geometric setting}

Let $J$ be a Euclidean Jordan algebra. The main reference for results and notation is \cite{fk}. See also \cite{s}.

\subsection{Complete system of orthogonal idempotents and the associated subalgebra}

\begin{definition}
A \emph{complete system of orthogonal idempotents} (CSOI for short) of $J$ is a family $\mathbf c = (c_1,c_2,\dots, c_k)$ of mutually orthogonal idempotents of $J$ such that $e=c_1+c_2+\dots+c_k$. 
\end{definition}

Let $\mathbf c = (c_1,c_2,\dots, c_k)$ be a CSOI of $J$. For each $j, 1\leq j\leq k,$ let $J_j$ be the Euclidean Jordan subalgebra defined by
\[J_j = J(c_j,1) = \{ x\in J, c_j x = x\}\ .
\]
Each $J_j$ is a subalgebra of $J$, and for any $i,j$ such that $1\leq i\neq j\leq k$, $J_i\cap J_j=\{0\}$ and  $J_iJ_j=0$. Define
\[J(\mathbf c) = \bigoplus_{j=1}^k J_j\ .
\]
Then $J(\mathbf c)$ is a Euclidean Jordan subalgebra of $J$, and we refer to it as the \emph{subalgebra associated to the CSOI $\mathbf c$}.

\subsection{The group $L(\mathbf c)$}

Let $\Str(J)$ be the structure group of $J$. Its  elements may be characterized as follows  : an element $\ell\in GL(J)$ belongs to $\Str(J)$ if and only if $\ell$ preserves  $J^\times$ (the open set of invertible elements in $J$) and there exists $h\in GL(J)$ such that for any $x\in J^\times$,
\[(\ell x)^{-1} =hx^{-1}\ .
\]
Moreover, the element $h$ is unique and equal to ${\ell^t}^{-1}$. 

Consequently, the structure group $\Str(J)$ is a closed Lie subgroup of $GL(J)$, stable by the Cartan involution $\ell\longmapsto {\ell^t}^{-1}$.

Let $\Omega$ be the symmetric cone in $J$ which can be defined as the set of squares of invertible elements. The group $G(\Omega)$ of all linear transformations of $J$ which preserve $\Omega$ is closely connected to $\Str(J)$. In fact, both groups have the same neutral component, henceforth denoted by $L$. The next result is introduced (with proof) because of lack of reference.
\begin{proposition}\label{Pomega}
 Let $J$ be a Euclidean Jordan algebra. The closed subgroup of $GL(J)$ generated by $\{P(x), x\in \Omega\}$ is equal to $L$.
\end{proposition}
\begin{proof} Let $L_1$ be the closed subgroup generated by $\{P(x), x\in \Omega\}$. For $x\in \Omega$, $P(x)$ belongs to $L$, so that $L_1\subset L$. 

Let $\mathfrak l= \Lie(L)= \mathfrak{str}(J)$ and $\mathfrak l_1 = \Lie(L_1)$. Let $\mathfrak p = \{ L(x), x\in J\}$. Recall that  for any $x\in J$ $P(\exp x) = \exp 2L(x)$. Moreover, $\Omega = \exp J$, so that $\mathfrak l_1\supset \mathfrak p$, and hence 
\[{\mathfrak l}_1\supset [{\mathfrak p}, {\mathfrak p} ] \oplus \mathfrak p\ .\]
On the other hand, 
\[\mathfrak l =  \Der(J) \oplus \mathfrak p\]
where $\Der(J)$ is the space of derivations of $J$, which is known to be equal to $ [\mathfrak p, \mathfrak p] $, so that
\[ \mathfrak l=[\mathfrak p, \mathfrak p] \oplus \mathfrak p\ .
\]
and hence $ {\mathfrak l}_1= {\mathfrak l} $. As both $L$ and $L_1$ are connected and $L_1\subset L$, the conclusion follows.
\end{proof}
Let $J$ be a Euclidean Jordan algebra, let $\mathbf c=(c_1,c_2,\dots, c_k)$ be a CSOI and let $J(\mathbf c)$ be the associated subalgebra. Define 
\[\Str(\mathbf c) = \{ \ell \in \Str(J), \ell \big(J(\mathbf c)\big) = J(\mathbf c)\} \ .
\]
The group $\Str(\mathbf c)$ is a Lie subgroup and its Lie algebra 
 $\mathfrak{str}(\mathbf c)$ is given by
\[\mathfrak{str}(\mathbf c) = \{ T\in \mathfrak{str}(J), T\big(J(\mathbf c)\big) \subset \big(J(\mathbf c)\big)\}\ .
\]
\begin{proposition}\label{cCartan}
The group $\Str(\mathbf c)$ is stable by the Cartan involution.
\end{proposition}
\begin{proof}
Let $J^\times$ be the open subset of invertible elements in $J$, and let $J(\mathbf c)^\times$ the open subset of invertible elements of $J(\mathbf c)$. Then
 \begin{equation}\label{invJc}
 J(\mathbf c)^\times =J(\mathbf c) \cap J^\times\ .
 \end{equation} 
 In fact, if $x\in J(\mathbf c)$ is invertible in $J$, its inverse belongs to $\mathbb R[x]$. As $e$ and $x$ belong to $J(\mathbf c)$, $\mathbb R[x]\subset J(\mathbf c)$, so that $x$ is invertible in $J(\mathbf c)$. Hence $J(\mathbf c) \cap J^\times \subset J(\mathbf c)^\times$. The opposite inclusion $J(\mathbf c)^\times \subset J^\times$ is trivial.
 
Now assume that $\ell$ belongs to $\Str(\mathbf c)$. Let $x\in J(c)\cap J^\times$. As ${\ell^t}^{-1}$ belongs to $\Str(J)$, 
\[{\ell^t}^{-1}x=( \ell x)^{-1}\in J^\times \cap J(\mathbf c)
\]
and n\eqref{invJc}  implies that ${\ell^t}^{-1}$ maps $J(\mathbf c)^\times$ into itself. As $J(\mathbf c)^\times $ is dense in $J(\mathbf c)$, ${\ell^t}^{-1}$ maps $J(\mathbf c)$ into itself. Hence the group $\Str(\mathbf c)$ is stable by the involution $\ell\longmapsto {\ell^t}^{-1}$.\end{proof}
 Notice that, as a consequence of Proposition \ref{cCartan}, the Lie algebra $\mathfrak str(\mathbf c)$ is stable by the Cartan involution $T\longmapsto -T^t$.

\begin{proposition}\label{TTj}
Let $T\in \mathfrak{str}(\mathbf c)$. Then $T$ maps each $J_j$ into itself. 
\end{proposition}
\begin{proof} If $T\in \mathfrak l(\mathbf c)$, then both $\frac{1}{2}(T-T^t)$ and $\frac{1}{2}(T+T^t)$ belong to $\mathfrak l(\mathbf c)$, so that it suffices to prove Proposition \ref{TTj} separately for $T = D\in \mathfrak l (\mathbf c) \cap \Der(J)$ and for those $T= L(v), v\in J$ which belong to $ \mathfrak l (\mathbf c)$.

 So let $D$ be a derivation of $J$ which maps $J(\mathbf c)$ into itself. For $j,1\leq j\leq k$,  $D c_j = D c_j^2=2 c_jDc_j$ so that 
$Dc_j\in J(c_j,\frac{1}{2})$. As $J(c_j,\frac{1}{2})\cap J(\mathbf c) = \{ 0\}$, $Dc_j=0$. Now for  $x\in J_j$,  $Dx=D(c_jx) = c_jDx$ and hence $Dx\in J_j$. As this is true for any $j$, the conclusion follows in this case.

 Next, let $T = L(v)$ for some $v\in J$  and assume that $L(v)$ maps $J(\mathbf c)$ into itself. Let $j, 1\leq j\leq k$ and let $v= v_1+v_{\frac{1}{2}}+v_0$ be its decomposition relative to the idempotent $c_j$. Then
 \[ L(v) c_j = L(c_j) v = v_1\ +\ \frac{1}{2}  \ v_{\frac{1}{2}}\ .
 \]
 As $J(c_j,\frac{1}{2}) \cap J(\mathbf c) = \{ 0\}$, this forces $L(v) c_j = v_1\in J_j$. So, for any $j,1\leq j\leq k$, $L(v)c_j\in J_j$. Hence
   \[v= L(v)e = L(v)(c_1+\dots+c_k) = L(v) c_1+\dots +L(v) c_k
 \] belongs to $J(\mathbf c)$. But now for $v\in J(\mathbf c)$, $L(v)$ maps each $J_j$ into itself and the conclusion follows.
\end{proof}
Let $L(\mathbf c)$ be the neutral component of $\Str(\mathbf c)$.

\begin{proposition}Let $\ell\in L(\mathbf c)$. Then $\ell$ maps $J_j$ into itself for any $j, 1\leq j\leq k$. For each $j, 1\leq j\leq k$, the induced map $\ell_j:J_j\longrightarrow J_j$ belongs to $L_j$, the connected component of $\Str(J_j)$.
\end{proposition}
\begin{proof} Let $\ell_t = \exp tT$ be a one-parameter subgroup in $L(\mathbf c)$. Then by differentiation at $t=0$,  $X$ belongs to $\mathfrak l(\mathbf c)$, which by Proposition \ref{TTj} implies $X(J_j)\subset J_j$ for any $j, 1\leq j\leq k$. Hence $\ell_t(J_j)\subset J_j$ for any $t\in \mathbb R$. So the proposition is satisfied for all elements of $L(\mathbf c)$ sufficiently closed to the identity. As $L(\mathbf c)$ is connected, the property $\ell(J_j) \subset J_j$ is valid for all $\ell\in L(\mathbf c)$. 

For the second part of the proposition, recall that for $\ell \in \Str(J)$ and $x\in J$,
\begin{equation}\label{gPg}
P(\ell x) = \ell P(x)\ell^t
\end{equation}
Let $\ell\in L(\mathbf c)$ and let $x\in J_j$, so that $\ell x_j\in J_j$. The operators $\ell, \ell^t, P(x)$ map $J_j$ into itself, so that by \eqref{gPg}, $P(\ell x)$ also maps $J_j$ into itself. By restriction to $J_j$, \eqref{gPg} implies
\[P_j(\ell_jx) = \ell_jP_j(x)(\ell^t)_j
\]

By \cite{fk} Lemma VIII.2.3 applied to $J_j$, this implies $\ell_j\in \Str(J_j)$ and $(\ell^t)_j = \ell_j^t$. As the restriction map $\ell\longmapsto \ell_j$ is continuous, it follows that $\ell_j$ belongs to $L_j$.
\end{proof}
The last proposition map allows to define the \emph{restriction map} 
\[L(\mathbf c)\ni \ell \ \longmapsto (\ell_1,\ell_2,\dots, \ell_k)\ \in L_1\times L_2\times \dots\times L_k\ .
\] 
For each $j, 1\leq j\leq k$ let $\Omega_j$ be the symmetric cone of $J_j$.  

\begin{proposition}\label{sumcones}
 Let $x_1\in J_1,\dots, x_j\in J_j,\dots,x_k\in J_k $ and let  $x= x_1+x_2+\dots +x_k\in J(\mathbf c)$. Then $x$ belongs to $\Omega$ iff  $x_j$ belongs to $\Omega_j$  for $ 1\leq j\leq k$.
\end{proposition}
\begin{proof} 
For any $j, 1\leq j\leq k$ let $r_j$ be the rank of $J_j$. There exists  a Jordan frame  $(e_j^{(1)}, \dots, e_j^{(r_j)})$ of $J_j$ such that 
\[x_j=a_j^{(1)} e_j^{(1)} +\dots + a_j^{(r_j)} e_j^{(r_j)}\ .
\] 
for some $a_j^{(i)}\in \mathbb R, 1\leq i\leq r_j$.
The collection $\big\{e_1^{(1)},\dots, e_1^{(r_1)}, \dots, e_k^{(1)},\dots e_k^{(r_k)}\big\}$ is a Jordan frame of $J$ and 
\[x= x_1+x_2+\dots+x_k =\sum_{j=1}^k \ \sum_{i=1}^{r_j} a_j^{(i)} e_j^{(i)} \ .\]
 Now $x_j$ belongs to $\Omega_j$ if and only if $a_j^{(i)}>0$ for $1\leq i\leq r_j$, and $x$ belongs to $\Omega$ if and only if $a_j^{(i)} >0$ for $1\leq j\leq k$ and $1\leq i\leq r_j$. The proof of the equivalence of the two properties follows easily.
\end{proof}

Consequently, let
\[\Omega(\mathbf c) =J(\mathbf c) \cap \Omega =  \Omega_1+\dots \Omega_{j}+\dots +\Omega_k .\]

\begin{proposition}\label{x+y}
 Let $c$ an idempotent of $J$ and let $x\in J(c,1),\ y\in J(c,0)$. Then $P(x+y)$ maps $J(c, \lambda)$ into itself ($\lambda=1,\frac{1}{2},0$) and is equal to
\[\begin{matrix} i)&\ P_1(x) &\text{ on } J(c,1))\\ ii)& 2\big(L(x)L(y)+L(y)L(x)\big) &\text{ on } J\big(c,\frac{1}{2}\big)\\  iii)&\  P_0(y) &\text{ on }J_0(c)\end{matrix}\ .
\] 
\end{proposition}
\begin{proof} As $P(y) = 2L(y)^2-L(y^2)$ we get
\[P(x+y)=\]\[2L(x+y)^2-L\big((x+y)^2\big)
\]
\[=2L(x)^2+2L(y)^2+2\big(L(x)L(y)+L(y)L(x)\big) -L(x^2)-L(y^2)-2L(xy)
\]
\[=P(x) +2\big(L(x)L(y)+L(y)L(x)\big) +P(y)\ .
\]
Now, as $P(c)x=x$, $P(x) = P\big(P(c)x\big) = P(c)P(x)P(c)$, $P(x)$ maps $J(c,1)$ on itself and is $0$ on $J(c,\frac{1}{2})\oplus J_0(c)$. Permuting the role of $c$ (resp. $x$) and $(e-c)$ (resp. $y$), $P(y)$ maps $J(c,0)$ on itself and is $0$ on $J(c,1)\oplus J(c,\frac{1}{2})$. Next, $L(x)$ is $0$ on $J(c,1)$ and $L(y)$ is $0$ on $J(c,1)$, so that $2(L(x)L(y)+L(y)L(x))$ is $0$ on $J(c,1)\oplus J(c,0)$, and finally both $L(x)$ and $L(y)$ maps  $J(c,\frac{1}{2})$ into itself, so that $2(L(x)L(y)+L(y)L(x))$ maps $J(c,\frac{1}{2})$ into itself. This completes the proof of Proposition \ref{x+y}.
\end{proof}

\begin{proposition}\label{Psum}
 Let $x_1\in \Omega_1,\dots, x_k\in \Omega_k$ and set $x=x_1+x_2+\dots +x_k$. Then $P(x)$ belongs to $L(\mathbf c)$ and its image by the restriction map is equal to
\[\big(P_1(x_1), P_2(x_2),\dots, P(x_k)\big)\ .
\]
\end{proposition}
\begin{proof} As $x$ belongs to $\Omega$, $P(x)$ belongs to $L$. Let $j, 1\leq j\leq k$ and set $y_j =x-x_j$. Then $x_j\in J(c_j,1)$ and $y_j\in J(c_j,0)$. Hence, by Proposition \ref{x+y},
$P(x) = P(x_j+y_j)$ maps $J(c_j,1)=J_j$ into itself and the induced restriction  on $J(c_j,1) = J_j$ is equal to $P_j(x_j)$.
\end{proof}
Let $K$ be the neutral component of the group of automorphisms (for the Jordan structure) of $J$. Recall that $K$ is a maximal compact subgroup of $L$ and its Lie algebra $\mathfrak k$ is equal to $ \Der(J)$. It is also the stabilizer of the neutral element $e$ in $L$. For each $j, 1\leq j\leq k$, the definition of $K_j$ is the same, adapted to the Jordan algebra $J_j$.
\begin{proposition}\label{surjK}
 The image of $K\cap L(\mathbf c)$ by the restriction map is equal to $K_1\times K_2\times \dots\times K_k$.
\end{proposition}
\begin{proof} Let $k\in K\cap L(\mathbf c)$ and let $(k_1,k_2,\dots,k_k)$ be its image by the restriction map. Then $ke=e$, so that, for any $j$,  $k_jc_j=c_j$ and hence $k_j\in K_j$. Next, let $u_j,v_j$ be in $J_j$. Then $[L(u_j),L(v_j)]$ is a derivation of $J$ which preserves $J(\mathbf c)$ and induces 
\[(0,\dots, 0, [L_j(u_j),L_j(v_j)],0,\dots,0)\]
on $J(\mathbf c)$. As elements of the form $[L(u_j),L(v_j)]$ generate $\Der(J_j)$, it follows that $\Der(J_1)\times \Der(J_2)\times \dots\times \Der(J_k)$ is contained in the image by the restriction map of $\Der (J)\cap L(\mathbf c)$. Hence the image by the restriction map of $K\cap L(\mathbf c)$ contains a neigbourhood of the neutral element in $K_1\times K_2\times \dots \times K_k$. As the image of $K\cap L(\mathbf c)$ is compact and as 
$K_1\times K_2\times \dots \times K_k$ is connected, the image is equal to $K_1\times K_2\times \dots \times K_k$.
\end{proof}
Let
\[M(\mathbf c) = \{ \ell \in L(\mathbf c), \ell x = x \text{ for all } x\in J(\mathbf c)\}\ .
\] 
Observe that $M(\mathbf c)$ is a closed normal subgroup of $L(\mathbf c)$, which is contained in $\Aut(J)$. Let $\Lie\big(M(\mathbf c)\big) = \mathfrak m(\mathbf c)$.
\begin{proposition}\label{Lsurj} {\ }
\smallskip

 $i)$ The restriction map 
\[ L(\mathbf c) \ni \ell \longmapsto (\ell_1,\dots, \ell_k)\in L_1\times\dots \times L_k
\]
is a surjective homomorphism. 
\smallskip

$ii)$ The kernel of the restriction map is equal to $M(\mathbf c)$.
\smallskip

$iii)$ $L(\mathbf c)$ is the closed subgroup of $L$ generated by $\{P(x),\ x\in \Omega(\mathbf c)\}$ and $M(\mathbf c)$\ .
\end{proposition}

\begin{proof} Any element of $L_j$ can be written as $k_jP(x_j)$ for some $k\in K_j$ and some $x_j\in \Omega_j$. So $i)$ follows easily from Proposition \ref{Psum} and Proposition \ref{surjK}. Further, $ii)$ is a merely a rephrasing of the definition of $M(\mathbf c)$.

Finally, let $x_j\in \Omega_j$ and set $\widetilde x_j = c_1+\dots+c_{j-1}+x_j+\dots +c_k$. The image of $P(\widetilde x_j)$ by the restriction map is equal to
\[(\Id_1,\dots, \Id_{j-1},P_j(x_j),\Id_{j+1},\dots, \Id_k)\ .
\]
As $L_j$ is the closed subgroup generated by $\{P_j(x_j), x_j\in \Omega_j\}$, the image by the restriction map of the closed subgroup generated by $\{P(\widetilde x_j), \widetilde x_j\in \Omega_j\}$ is equal to $\{\Id_1\}\times \dots\times \{\Id_{j-1}\}\times L_j\times \dots \times \{\Id_k\}$ and $iii)$ follows by repeating the argument for all $j, 1\leq j\leq k$.
\end{proof}

\subsection{The group $G(\mathbf c)$}

Let $\mathbb J$ be the complexification of $J$, and let
\[T_\Omega = J+i\Omega \subset \mathbb J
\]
be the tube over the cone $\Omega$, which is a Hermitian symmetric domain. Let $G$ be the neutral component of the group $\Aut(T_\Omega)$ of all bi-holomorphic automorphisms of $T_\Omega$.

Let $\mathbb J_j$ (resp. $\mathbb J(\mathbf c)$) be the complexification of  $J_j$ (resp. $J(\mathbf c)$).
For $j, 1\leq j\leq k$ consider the tube-type domain $T_{\Omega_j}\subset \mathbb J_j$. As a consequence of Proposition \ref{sumcones},
\[T_\Omega \cap J(\mathbf c)  = T_{\Omega_1} \oplus \dots \oplus T_{\Omega_j}+\dots \oplus T_{\Omega_k},
\]
and this domain will be denoted by $T_{\Omega(\mathbf c)}$.

Let $G(\mathbf c)$ be the neutral component of the subgroup of $G$ defined by
\[\{ g\in G,\quad  g(T_{\Omega(\mathbf c)})\subset T_{\Omega(\mathbf c)}\}\ .
\]
The subgroup $G(\mathbf c)$ is a closed Lie subgroup of $G$. We first determine the Lie algebra $\mathfrak g(\mathbf c)$ of $G(\mathbf c)$.
Recall that the Lie algebra of $\mathfrak g$ is given by  
\[\mathfrak g = J\oplus \mathfrak l\oplus J= \{X=(u,T,v) ,\quad  u\in J, \quad T \in \mathfrak l,\quad  v\in J\}
\] with Lie bracket
\[[(u_1,T_1,v_1),(u_2,T_2,v_2) ] = (u,T,v)
\]
\begin{equation}
\begin{split}
u=&T_1u_2-T_2 u_1,\\ T= &[T_1,T_2] +2u_1\square v_2-2 u_2\square v_1,\\\ v=& -T_1^tv_2+T_2^t v_1 \ .
\end{split}
\end{equation}
For $X=(u,T,v)\in \mathfrak g$, the vector field induced on $T_\Omega$ by the one-parameter subgroup $\exp tX$ is equal to
\[X(z) = u+Tz-P(z) v\ ,\qquad z\in T_\Omega\ .
\]
See \cite{fk} X.5 page 211.
\begin{proposition}
The Lie algebra of $G(\mathbf c)$ is given by
\[\mathfrak g(\mathbf c) = \{ X=(u,T, v), \qquad u\in J(\mathbf c),\quad T\in \mathfrak{l}(\mathbf c),\quad   v\in J(\mathbf c)\} .
\]
\end{proposition}
\begin{proof}
Assume that  $X=(u, T,v)$ with $u\in J(\mathbf c), T\in \mathfrak l(\mathbf c), v\in J(c)$. Let $z\in T_{\Omega(c)}$. Then, by an elementary calculation  $u+Tz+P(z)v$ belongs to $\mathbb J(\mathbf c)$ and hence, by exponentiation, $\exp tX$ preserves $T_{\Omega(\mathbf c)}$ so that $X\in \mathfrak g(\mathbf c)$.

Conversely, suppose that the vector field $X(z), z\in T_\Omega$ is associated to a one-parameter group of $G(\mathbf c)$. Then $X(z)= u+Tz-P(z)v$ must belong to $\mathbb J(\mathbf c)$ whenever $z$ belongs to $T_{\Omega(\mathbf c)}$. As $X(z)$ is polynomial in $z$, this property has to be valid for any  $z\in \mathbb J(\mathbf c)$. Set $z=0$ to get $u\in J(\mathbf c)$. Next, for any $t\in \mathbb C^\times$ and $z\in \mathbb J(\mathbf c)$
\[\frac{1}{t} \big(X(tz)-u\big) = Tz-t P(z) v\in \mathbb J(\mathbf c) \ .
\]
As this has to be valid for any $t\in \mathbb C^\times$,  necessarily, $Tz\in \mathbb J(\mathbf c)$ and $P(z) v\in \mathbb J(\mathbf c)$ for any $z\in \mathbb J(\mathbf c)$. This implies that $T$ maps $J(\mathbf c)$ into itself, hence belongs to $\mathfrak l 
 (\mathbf c)$. Finally, set $z=e\in \mathbb J(\mathbf c)$ to get $P(e)v = v\in J(\mathbf c)$.
\end{proof}

\begin{proposition}\label{GGj}
Let $g\in G(\mathbf c)$. Then there exists $g_1\in G_1,\dots, g_k\in G_k$ such that for any $z=z_1+z_2+\dots+z_k\in T_{\Omega(\mathbf c)}$,
\begin{equation}\label{ggj}
g(z) = g_1(z_1)+g_2(z_2)+\dots+g_k(z_k)\ .
\end{equation}
\end{proposition}
\begin{proof} Let first assume that $g=\exp X$, where $X\in \mathfrak g(\mathbf c)$. Then the vector field generated by the one parameter group $g_t=\exp tX$ is given by
\[X(z) = u+Tz-P(z)v\ , u,v\in J(\mathbf c), T\in \mathfrak l(\mathbf c)\ .
\]
More explicitly, let
\[u=u_1+u_2+\dots u_k,\qquad v = v_1+v_2+\dots+v_k, \]
where for any $j, 1\leq j\leq k$, $ u_j,v_j\in J_j $.  By Proposition \ref{TTj}, $T$ maps each $J_j$ into itself and induces for each $j, 1\leq j\leq k$ an endomorphism $T_j$, which belongs to $\mathfrak l_j$. As a consequence, at any point $z=z_1+z_2+\dots+z_k\in T_{\Omega(\mathbf c)}$, 
\[X(z) = X_1(z_1)+X_2(z_2)+\dots + X_k(z_k)
\]
where $X_j(z_j) = u_j+T_j z_j-P(z_j)v_j$. Now let $X_j= X(u_j,T_j,v_j)\in \mathfrak g_j$ and, for $t\in \mathbb R$,  let $g_{j,t}=\exp tX_j\in G_j$. Then by integration
\[g_t(z_1+z_2+\dots +z_k) = g_{1,t}(z_1)+g_{2,t}(z_2)\dots +g_{k,t}(z_k)\ 
\] 
for any $z=z_1+z_2+\dots z_k\in T_{\Omega(\mathbf c)}$. So, the existence of a family $(g_1,\dots, g_k)$ which satisfies \eqref{ggj} is proven for any $g$ in the image of the exponential map, in particular in some neighborhood of  the identity in $G(\mathbf c)$.  Now property \eqref{ggj} is stable by composition, i.e. if satisfied by two elements $g$ and $h$ of $G(\mathbf c)$, it is satisfied for $gh$. As $G(\mathbf c)$ is connected, the property \eqref{ggj} is satisfied for all elements of $G(\mathbf c)$. This finishes the proof of the proposition.
\end{proof}
The last proposition allows to define the \emph{restriction map} 
\[G(\mathbf c) \ni g \longmapsto (g_1,g_2,\dots, g_k) \in G_1\times G_2\times \dots \times G_k\  .
\] 
\begin{theorem}\label{surjkerrest}
 The restriction map is surjective. Its kernel is equal to $M(\mathbf c)$.
\end{theorem}
\begin{proof}
The restriction map induces an homomorphism from $\mathfrak g(c)$ into $\mathfrak g_1\times \mathfrak g_2\times \dots \times \mathfrak g_k$. From the proof of Proposotion \ref{GGj}, this homomorphism is given by

\[(X=(u,T,v)\longmapsto \big( (u_1,T_1,v_1),\dots, (u_j,T_j,v_j),\dots, (u_k,T_k,v_k)
\] 
where $u=u_1+u_2+\dots+u_k, u_j\in J_j$, $v=v_1+\dots+v_k, v_j\in J_j$ and $T_j = T_{\vert J_j}$. By (the Lie algebra version of) Proposition \ref{Lsurj}, this homomorphism is surjective. Hence the image of restriction map contains a neighborhood of  the neutral element in $G_1\times \dots \times G_k$. As for any $j, 1\leq j\leq k$, $G_j$ is connected,  this implies the surjectivity.

Now let $g$ be an element of $G(\mathbf c)$ which is in the kernel of the restriction map. Then, as $c_1+c_2+\dots c_k=e$, the point $ie$ belongs to $T_{\Omega(\mathbf c)}$ and hence by assumption, $g(ie)= ie$. The differential $Dg(ie)$ of $g$ at $ie$ belongs to $L$. The tangent space to $T(\Omega(\mathbf c))$ at $ie$ is equal to  $\mathbb J(\mathbf c)$, and $Dg(ie)$  induces the identity on it, hence belongs to $M(\mathbf c)$, say $Dg(ie) = m$ with $m\in M(\mathbf c)$. Now consider $g'=g\circ m^{-1}$. Then $g'$ is an isometry of the tube $T_\Omega$ for the Bergman metric which fixes the point $ie$ and has differential $Dg'(ie)=\id_{\mathbb J}$. Hence $g'=\id$ and so $g\in M(\mathbf c)$.
\end{proof}

\subsection{$\mathbf c$-homogeneous polynomials and $L(\mathbf c)$-covariant differential operators}

In this section we assume that $J$ is a \emph{simple} Euclidean Jordan algebra.
\begin{proposition}\label{simpleJc}
 Let $J$ be a simple Euclidean Jordan algebra and let $c$ be an idempotent of $J$. Then $J(c,1)$ is a simple Jordan algebra.
\end{proposition}
\begin{proof} Let $c=d_1+d_2+\dots+d_m$ be a decomposition of $c$ as a sum of primitive  idempotents of $J$. For any $1\leq j\leq m$, as $cd_j= d_j$, $d_j$ belongs to $J(c,1)$. This family of primitive orthogonal idempotents can be completed to obtain a Jordan frame of $J$, namely $\{d_1,d_2,\dots, d_m,d_{m+1}, \dots, d_r\}$. The corresponding Peirce decomposition yields

\[ J = \bigoplus_{1\leq i\leq j\leq r}J_{ij} .
\]
Now by an elementary check,
\[J_{ij} \cap J(c,1) \neq \{0\} \text{ if and only if } 1\leq i\leq j\leq m\ .
\]
Hence $J(c) = \bigoplus_{1\leq i\leq j\leq k} J_{ij}$. Now assume that $J(c,1)$ could be decomposed as a sum of two non trivial Jordan subalgebras $J(c,1)=J_1\oplus J_2$. Let $e_1$ (resp. $e_2$) be the neutral element of $J_1$ (resp. $J_2$). Then $c=e_1+e_2$ is a decomposition of $c$ as a sum of two orthogonal idempotents, and there is a Jordan frame of $J(c,1)$ say $\{d_1,d_2,\dots, d_m\}$ such that 
\[e_1=d_1+\dots+d_p,\qquad  e_2 = d_{p+1}+\dots+d_m\ .
\]
Notice that $d_1,\dots, d_p\in J_1, d_{p+1},\dots,d_m\in J_2$.
The space $V_{1,p+1}$ is not reduced to $\{0\}$ (a consequence of the fact that $J$ is assumed to be simple). Let $w \in V_{1,p+1}, w\neq 0$. Then $w=w_1+w_2, w_1\in J_1, w_2\in J_2$. Now $d_1 w_1 = d_1 w = \frac{1}{2} w$, hence $w\in J_1$. But similarly $d_{p+1} w_2 = \frac{1}{2} w$ and hence $w\in J_2$. As $w\neq 0$ whereas $J_1\cap J_2 = \{ 0\}$, this yield a contradiction. Q.E.D.
\end{proof}

Let $\mathbf c=(c_1,c_2,\dots, c_k)$ be a complete system of idempotents. and let $J(\mathbf c)$ be the associated algebra. As a consequence of Proposition \ref{simpleJc},   $J_j=J(c_j,1)$ is simple for any $j,1\leq j\leq k$. 

There exists a character  $\chi$ on $L$, such that 
\begin{equation}\label{defchi}
\text{ for } \ell\in L,\quad x\in J,\qquad  \det(\ell x) = \chi(\ell) \det x\ ,
\end{equation}
and similarly for any $j, 1\leq j\leq k$ 
\[\text{ for } \ell_j\in L_j,\quad\ x\in J_j,\qquad {\det}_j(\ell_j x_j) = \chi_j(\ell_j) {\det} _jx_j\ .
\]
\begin{proposition}\label{chichij}
 Let $\ell \in L(\mathbf c)$ and let $(\ell_1,\ell_2,\dots, \ell_k)$ be its restriction to $J(\mathbf c)$. Then
\[\chi(\ell) = \chi_1(\ell_1)\chi_2(\ell_2)\dots \chi_k(\ell_k)\ .
\]
\end{proposition}
\begin{proof} Let $x= x_1 +x_2+\dots +x_k$ be in $\in J(\mathbf c)$. Then
\[\det x = {\det}_1x_1\, {\det}_2 x_2\,\dots \,{\det}_k x_k \ .\]
Similarly, $\ell x = \ell_1 x_1+\ell_2 x_2+\dots+\ell_k x_k$ and
\[ \det \ell x={ \det}_1 \,(\ell_1 x_1){ \det}_2 (\ell_2 x_2)
 \dots { \det}_k\, (\ell_k x_k)\ .
\]
and the statement follows easily by choosing $x$ such that $\det x\neq 0$.

\end{proof}
\begin{definition}
A polynomial $q$ on $J$ is said to be \emph{$\mathbf c$-homogenous of degree $(p_1,p_2,\dots,p_k)$}, if  for any $\ell \in L(\mathbf c)$,  for any $x\in J$
\begin{equation}\label{covLq}
 q(\ell x) = \chi_1(\ell_1)^{p_1} \dots \chi_k(\ell_k)^{p_k} q(x)\ ,
\end{equation}
where $(\ell_1,\ell_2,\dots, \ell_k)$ is the restriction of $\ell$ to $J(\mathbf c)$.
\end{definition}
Let $q$ be a polynomial on $J$, which we extend as a holomorphic polynomial on  $\mathbb J$. Let $D_q$ be the unique constant coefficients holomorphic differential operator on $\mathbb J$ such that for any $v\in \mathbb J$
\begin{equation}\label{construcD}
D_q \,e^{(z,v)}=  q(v) \,e^{(z,v)}\ .
\end{equation}
The polynomial $q$ is said to be the \emph{algebraic symbol} of $D_q$.

\begin{proposition} Let $q$ be a polynomial on $J$ and assume that $q$ is $\mathbf c$-homogeneous of multidegree $(p_1,p_2,\dots, p_k)$. Then the differential operator $D_q$ satisfies for all $f\in C^\infty(\mathbb J)$
\begin{equation}\label{covLDq}
D_q(f\circ \ell^{-1}) = \prod_{j=1}^k \chi_j(\ell_j)^{-p_j}(D_qf)\circ \ell^{-1}\ .
\end{equation}
where $\ell\in L(\mathbf c)$ and  $(\ell_1,\ell_2,\dots, \ell_k)$ is its restriction to $J(\mathbf c)$.
\end{proposition} 
\begin{proof} For $v\in \mathbb J$ let $f_v$ be the function defined on $\mathbb J$ by
\[f_v(z) = e^{(z,v)}
\]
It is enough to prove \eqref{covLDq} for the family of functions $(f_v), v\in \mathbb J$. Now
$f_v\circ \ell^{-1} = f_{{\ell^{-1}}^t v}$, so that
\[D_q (f_v\circ \ell^{-1})= D_q f_{{\ell^{-1}}^t v}=  q({\ell^{-1}}^t v)  f_{{\ell^{-1}}^t v}
\]
whereas
\[D_qf_v= q(v) f_v,\quad D_qf_v\circ \ell^{-1} = q(v)  f_{{\ell^{-1}}^t v}\ .
\]
Use the covariance property \eqref{covLq} of $q$ to conclude.
\end{proof}

\section{Jordan algebra representations}

\subsection{Generalities}
 Recall that a \emph {representation} of $J$ is a triple $(E,\langle\,.\,,\,.\,\rangle,\Phi)$ where $(E, \langle\,.\,,\,.\,\rangle)$ is a finite dimensional Euclidean vector space and $\Phi:J\rightarrow \Sym(E)$ a unital Jordan algebra morphism, i.e. $\Phi$ satisfies 

$i)\ \Phi$ is linear
\smallskip

$ii)\ \Phi(xy) = \frac{1}{2}\big(\Phi(x)\Phi(y)+\Phi(y)\Phi(x)\big), \text{ for all } x,y\in J$
\smallskip

$iii)\ \Phi(e) = \id_E$
\smallskip

$iv)\ \langle\Phi(x) \xi,\eta\rangle = \langle\xi, \Phi(x) \eta\rangle, \text{ for all } x\in J,\  \xi, \eta\in E\ .$

For general results on the subject, see \cite{fk}, IV.4 and ch. XVI, and \cite{c92}. 

The assumptions imply that 
\begin{equation}
\Phi(x) \text{ is invertible if } x \text{ is invertible and then } \Phi(x)^{-1} = \Phi(x^{-1})\ ,\end{equation}
\begin{equation}\label{PhiP}
\Phi(P(x)y) = \Phi(x)\Phi(y)\Phi(x), \text{ for all }x,y\in J\ .
\end{equation}

To a representation $E$ of $J$ is associated a symmetric bilinear map $H : E\times E \rightarrow J$ define by
\[\text{ for all } x\in J,\qquad  \big(H(\xi,\eta),x\big)=  \langle \Phi(x)\xi, \eta\rangle \ .
\]
Denote by $Q: E\rightarrow J$ the associated quadratic map, defined by \[Q(\xi) = H(\xi,\xi)\ .\]

Let us recall some results which will be necessary later. For proofs, see \cite{c92}.
\begin{proposition} An element $\xi\in E$ is said to be \emph{regular} if one of the four equivalent  propositions is satisfied :

$i)$ $x\in J, \Phi(x)\xi = 0 \Longrightarrow x=0$
\smallskip

$ii)$ $Q(\xi) \in \Omega$
\smallskip

$iii)$ $\det Q(\xi) \neq 0$
\smallskip

$iv)$ the differential $d_\xi Q$ of $Q$ at $\xi$ is surjective.
\end{proposition} 

\begin{proposition} A representation $\Phi$ on $E$ is said to be \emph{regular} if one of the three equivalent propositions is verified :

$i)$ $\exists\, \xi \in E,\ Q(\xi) = e$
\smallskip

$ii)$ $\exists\, \xi \in E,\  \xi \text{ regular}$
\smallskip

$iii)$ $Q(E)  \supset \Omega$\ .
\end{proposition}
\begin{proposition}\label{JsimpleE}
 Let $J$ be a simple Euclidean Jordan algebra of rank $r$, and let $(E,\Phi)$ be a representation of $J$ of dimension $N$. 
\smallskip

$i)$ for a primitive idempotent  $d$ of $J$, $q: =\rank(\Phi(d))$ is independent of $d$ and $N=rq$.

$ii)$ for any idempotent  $c$ of rank $r_c$, $\rank(\Phi(c)) = r_cq$ \ .

$iii)$ for $x\in J,\qquad \Det \Phi(x)= (\det x)^\frac{N}{r}$.
\end{proposition}

\begin{proof} For $i)$ and $iii)$, see \cite{fk}. For $ii)$,  let $c=d_1+\dots +d_{r_c}$ be a Peirce decomposition of $c$ as a sum of mutually orthogonal idempotents. Then $\Phi(c) = \Phi(d_1)+\dots + \Phi(d_{r_c})$ is a sum of mutually orthogonal projectors, each of rank $q$. Hence $\rank\big(\Phi(c)\big) = r_cq$.
\end{proof}

Let us give two typical examples of representations.

Let $J=\Symm(r)$ the space of real symmetric matrices of size $r$, with the Jordan product $x.y = \frac{1}{2}(xy+yx)$. Let $q$ be an integer, $q\geq 1$ and let $E=\Mat(r,q)$ be the space of real matrices of size $(r,q)$, equipped with its standard inner product $\langle \xi,\eta\rangle= \tr (\xi\eta^t)$. For $x\in J$ and $\xi\in E$, define
\[\Phi(x) \xi = x\xi\ .
\]
Then $\Phi$  is a representation of $J$ on $E$. The representation is regular if and only if $q\geq r$. This example plays a fundamental r\^ole in Ibukiyama's work.

The second example will be studied in Section 7. Let $J$ be the quadratic algebra of dimension $n$, i.e. $J=\mathbb R \oplus \mathbb R^{n-1}$, with the Jordan product
\[(s,x) (t,y) = (st+x.y, sy+tx),\qquad s,t\in \mathbb R, x,y\in \mathbb R^{n-1}\ ,
\]
where $x.y= x_1y_1+\dots x_{n-1}y_{n-1}$.

Let $Cl(\mathbb R^{n-1})$ be the Clifford algebra of $\mathbb R^{n-1}$, generated by $\mathbb R^{n-1}$ and the relations\footnote{Notice that it differs   from the usual convention by the absence of a sign $-$}
\[x,y\in \mathbb R^{n-1},\qquad xy+yx= 2\ x.y\ .
\]

The representations of $J$ coincide with the \emph{Clifford modules}  for the algebra $Cl(\mathbb R^{n-1})$. In fact, If $E$ is a Clifford module, then for $x\in J$ and $\xi\in E$,
\[\Phi(s,x)\,\xi = s\,\xi + x\,\xi
\]
defines a representation of $J$ on $E$, and vice versa. The regularity of these representations is a rather subtle question, see \cite{c92}.

\subsection{Restriction to $J(\mathbf c)$}

We now consider the situation studied in Section 1. So let $\mathbf c = (c_1,c_2,\dots, c_k)$ be a CSOI of $J$.
\begin{proposition} The operators $\{\Phi(c_j), 1\leq j\leq k \}$ form a complete family of orthogonal projectors on $E$.
\end{proposition}

The proof is similar to the proof  given for the case of a Jordan frame of $J$ in \cite{fk} Section IV.4 or in \cite{c92}.

Let $E=E_1\oplus E_2\oplus \cdots \oplus E_k$ be the corresponding orthogonal decomposition of the space $E$. 
\begin{proposition} Let $j, 1\leq j\leq k$, and let $x\in J_j$. Then
\smallskip

$i)$ for $i\neq j$, $\Phi(x) E_i=0$
\smallskip

$ii)$ $\Phi(x) E_j\subset E_j$
\smallskip

$iii)$ The map $\Phi_j : J_j \longrightarrow \End(E_j)$ given by $\Phi_j(x) = \Phi(x)_{\vert E_j}$ yields a representation of the Jordan algebra $J_j$. 
\end{proposition}
\begin{proof} 

$i)$ Let $i\neq j$ and let $\xi_i\in E_i$. Then
\[\Phi(x) \xi_i= \Phi(c_jx)\xi_i = \frac{1}{2}\big(\Phi(c_j)\Phi(x) + \Phi(x) \Phi(c_j)\big)\xi_i=  \frac{1}{2}\Phi(c_j)\Phi(x) \xi_i\ .
\]
But $2$ is not an eigenvalue of $\Phi(c_j)$, so that $\Phi(x) \xi_i = 0$.
\smallskip

$ii)$ Let $\xi\in E_j$. Then
\[\Phi(x) \xi= \Phi(c_jx)\xi = \frac{1}{2}\big(\Phi(c_j)\Phi(x) + \Phi(x) \Phi(c_j)\big)\xi=  \frac{1}{2}\Phi(c_j)\Phi(x) \xi+ \frac{1}{2} \Phi(x) \xi\ ,
\]
so that
\[\Phi(x) \xi\ = \Phi(c_j)\Phi(x) \xi
\]
and hence $\Phi(x)\,\xi$ belongs to $E_j$.
\smallskip

$iii)$ The verification of this result is elementary.
\end{proof}

For each $j, 1\leq j\leq k$, let $Q_j : E_j\longrightarrow J_j$ be the quadratic map associated to the representation $E_j,\Phi_j$ on $E_j$.

Let $\proj(\mathbf c)$ be the orthogonal projection from $J$ onto $J(\mathbf c)$. 

\begin{proposition} For any $\xi=\xi_1+\xi_2+\dots+\xi_k\in E$,
\begin{equation}\label{projcQ}
\proj(\mathbf c) Q(\xi) = \sum_{j=1}^k Q_j(\xi_j)\ .
\end{equation} 
\end{proposition}
\begin{proof}
 Let $x=x_1+\dots+x_k\in J(\mathbf c)$ and let $\xi\in E$. Then
\[\big(x,Q(\xi)\big) = \sum_{j=1}^k \big(x_j,Q(\xi)\big)=  \sum_{j=1}^k\langle \Phi(x_j)\xi,\xi\rangle\] \[=  \sum_{j=1}^k\langle \Phi_j(x_j)\xi_j,\xi_j\rangle= \sum_{j=1}^k\big(x_j,Q_j(\xi_j)\big)\ .
\]
As $\sum_{j=1}^k Q_j(\xi_j)$ belongs to $J(\mathbf c)$, the statement follows.
\end{proof}

\begin{proposition} If the representation $(E,\Phi)$ is regular, then for any $j, 1\leq j\leq k$ the representation $(E_j,\Phi_j)$ of $J_j$ is regular.
\end{proposition}
\begin{proof} By assumption, there exists $\xi^0\in E$ such that $Q(\xi_0)= e$. Now, use \eqref{projcQ} to obtain
\[e= \proj(\mathbf c)e = \proj(\mathbf c) Q(\xi^0)= \sum_{j=1}^k Q_j(\xi^0_j),
\]
and hence for each $j, 1\leq j\leq k$
\[Q_j(\xi^0_j) = c_j\ ,
\]
so that the representation $(E_j, \Phi_j)$ is regular.
\end{proof}

\begin{proposition} Let $J$ be a \emph{simple} Euclidean Jordan algebra, and let $\mathbf c=(c_1,c_2,\dots, c_k)$ be a CSOI, and let $r_j$ be the rank of $J_j$. Let $(E,\Phi)$ be a representation of $J$, and let $N=rq$ be its dimension.
Then, for any $j,1\leq j\leq k$, 
\begin{equation}\label{NrNjrj}
\dim(E_j)= N_j= r_jq\ .
\end{equation}
\end{proposition}
\begin{proof} This is a direct application of Proposition \ref{JsimpleE}.
\end{proof} 

\begin{proposition} Let $q$ be a polynomial on $J$ which is $\mathbf c$-homgeneous of multidegree $(p_1,\dots, p_k)$.  Let $p$ be the polynomial on $E$ defined by $p=q\circ Q$. Then $p$ satisfies
for any $x=x_1+x_2+\dots+x_k\in J(\mathbf c)$, 
\begin{equation}\label{pchom}
p(\Phi(x) \xi) = \prod_{j=1} ^k {\det}_j( x_j)^{2p_j} p(\xi)\ .
\end{equation}
 \end{proposition}
\begin{proof}
Let $x=x_1+x_2+\dots+x_k\in \Omega(\mathbf c)$. Then $P(x)$ belongs to $L(\mathbf c)$ and its restriction to $E_j$  is equal to $P_j(x_j)$.  As $\chi_j((P_j(x_j)) ={\det}_j( x_j)^2$ follows
\[p(\Phi(x) \xi) = q\big( Q(\Phi(x) \xi)\big) = q\big(P(x) Q(\xi)\big)= \prod_{j=1} ^k {\det}_j( x_j)^{2p_j} q(Q(\xi))\ .
\]
As both sides of  \eqref{pchom} are polynomial in $x$, and $\Omega(\mathbf c)$ is an open set of of $J(\mathbf c)$, the identity is valid on $J(\mathbf c)$.
\end{proof}

\section{Pluriharmonic polynomials}

Let $J$ be a Euclidean Jordan algebra, $(E,\Phi)$ a representation of $J$, and let $Q:E\times E\longrightarrow J$ be the associated quadratic map. 
\begin{definition} A polynomial $p$ on $E$ is said to be \emph{pluri-harmonic} if for all $x\in J$
\begin{equation}
\Delta_E (p\circ\Phi(x)) = 0\ ,
\end{equation}
where $\Delta_E$ is the Laplacian on $E$.
\end{definition}
 
\subsection{The Hecke formula for pluri-harmonic polynomials}

A key result to be used later is a generalization of the classical Hecke formula, which we recall now, with a proof. For earlier occurrences of such results, see \cite{kv, i}. Let us also mention the paper \cite{c00bis}, where a similar result  in an even more general context is proved.

Let $J$ be a simple Euclidean Jordan algebra, and assume that $(E,\Phi)$ is a regular representation of $J$ of dimension $N$.

\begin{proposition} Let $p$ be a pluri-harmonic polynomial on $E$. Then for $x\in \Omega$, 
\begin{equation}
\int_E e^{i\langle \xi, \eta\rangle}e^{-\frac{1}{2}(x,Q(\xi))}\,p(\xi)\, d\xi = (2\pi)^{\frac{N}{2}}\, (\det x)^{-\frac{N}{2r}}\,e^{-\frac{1}{2} (x^{-1},\, Q(\eta))}\ p\big(i\phi(x^{-1}) \eta)\big)
\end{equation}
\end{proposition}
\begin{proof} First notice that for $x\in \Omega$, the quadratic form $(x,Q(\xi))=\langle \Phi(x)\xi, \xi\rangle=\langle \Phi(x^{\frac{1}{2}}) \xi,  \Phi(x^{\frac{1}{2}}) \xi\rangle$ is positive definite on $E$. Hence the integral on the left hand side converges. Next, recall that for any harmonic polynomial $q$ on a Euclidean space $F$,
\begin{equation}\label{FGharmonic}
\int_F e^{i\langle \xi, \eta\rangle}\, e^{-\frac{1}{2} \langle \xi, \xi\rangle}\,q(\xi) d\xi = (2\pi)^{\frac{N}{2}}\, e^{-\frac{1}{2}\langle \eta, \eta \rangle}\, q(i\eta)\ .
\end{equation}
Let $\xi'=\Phi(x^{\frac{1}{2}})\xi$. Then
\[\int_E e^{i\langle \xi, \eta\rangle}e^{-\frac{1}{2}(x,Q(\xi))}\,p(\xi)\, d\xi=
\]
\[\int_E e^{i\langle \xi',\Phi(x^{-\frac{1}{2}}\eta\rangle}e^{-\frac{1}{2}\langle \xi',\xi'\rangle}p(\Phi(x^{-\frac{1}{2}})\xi') \Det \Phi(x^{-\frac{1}{2}})\, d\xi'\ .
\]
\[= (2\pi)^{\frac{N}{2}} (\det x)^{-\frac{N}{2r}}e^{-\frac{1}{2} (x^{-1},\, Q(\eta))}\ p\big(i\phi(x^{-1}) \eta\big)\ ,
\]
where we  apply \eqref{FGharmonic} to $q=p\circ \Phi(x^{-\frac{1}{2}}), \eta'= \Phi(x^{-\frac{1}{2}})\eta$ and notice that by Proposition \ref{JsimpleE} $iii)$, $\Det \Phi(x^{-\frac{1}{2}})=(\det x)^{-\frac{N}{2r}}$.
\end{proof}
Let $\mathbb J$ be the complexified Jordan algebra, $\mathbb E$ the complexification of $E$, and extend $\Phi$ $\mathbb C$-linearly to $\mathbb J$. The previous formula has an extension to this complex setting.
\begin{proposition}\label{Fourier}
 Let $p$ be a pluriharmonic polynomial on $E$ and extend it holomorphically to $\mathbb E$. Then for any $z\in T_\Omega$,
\begin{equation}\label{FGCharmonic}
\int_E e^{i\langle \xi, \eta\rangle}e^{\frac{i}{2}(z,Q(\xi))}\,p(\xi)\, d\xi = (2\pi)^{\frac{N}{2}}\, \big(\det ( {\frac{z}{i}})\big)^{-\frac{N}{2r}}\,e^{\frac{i}{2} (-z^{-1},\, Q(\eta))}\ p\big(\phi(-z^{-1}) \eta)\big)\ ,
\end{equation}
where $\big(\det(\frac{z}{i})\big)^{-\frac{N}{2r}}$ is computed using the determination which is equal to $(\det y)^{-\frac{N}{2r}}$ for $z=iy, y\in \Omega$.
\end{proposition}
\begin{proof}
The left hand side of \eqref{FGCharmonic} is a  convergent integral depending holomorphically on the parameter $z\in T_\Omega$. The right handside is also holomorphic and both sides coincide on $i\Omega \subset T_\Omega$ by \eqref{FGharmonic}. The conclusion follows.
\end{proof}
\subsection{$\mathbf c$-pluri-harmonic polynomials}

Let $J$ be a simple Euclidean Jordan algebra, let $\mathbb c$ be a CSOI of $J$, and let $(E, \Phi)$ a representation of $J$.

\begin{definition} A polynomial $p$ on $E$ is said
 to be \emph{$\mathbf c$-pluri-harmonic} (with respect to $\Phi$) if for any $x\in J(\mathbf c)$ 
\begin{equation}\label{cph1}
\text{ for any } x\in J(\mathbf c),\qquad  \Delta_E \big(p\circ \Phi(x)\big)=0 \ .
\end{equation}
Equivalently, 

for any $ \xi_1\in E_1,\dots, \xi_{j-1}\in E_{j-1},\xi_{j+1}\in E_{j+1},\dots + \xi_k\in E_k$, 
\begin{equation}\label{cph2}
\begin{split}
\text {the polynomial }&\xi_j \longmapsto p(\xi_1,\xi_2, \dots, \xi_{j-1},\xi_j, \xi_{j+1}, \dots,\xi_k) \\
&\text{is pluriharmonic on } E_j \text{ (w.r.t. }\Phi_j).
\end{split}
\end{equation}
\end{definition}

In fact, assume that $p$ is $\mathbf c$-pluri-harmonic. Let $x_j\in J_j$. For $\xi=\xi_1+\dots \xi_k\in E$,
\[\Phi(c_1+\dots +c_{j-1}+x_j+c_{j+1}+\dots +c_k) (\xi)\]\[ = \xi_1+\dots +\xi_{j-1}+\Phi_j(x_j) \xi_j+\xi_{j+1}+ \dots + \xi_k,
\]
and hence the equivalent definition is satisfied.

Conversely, assume $p$ is a polynomial on $E$ which satisfies the conditions \eqref{cph2}. Let $x=x_1+\dots+x_k\in J(\mathbf c)$. For $\xi=\xi_1+\dots+\xi_k\in E$,
\[p(\Phi(x) \xi) = p\big(\Phi_1(x_1)\xi_1,\dots, \Phi(x_k)\xi_k\big)\ .
\] 

So, if $p$ satisfies the condition \eqref{cph2}, $\Delta_{E_j} (p\circ \Phi(x)) =0$ for any $j, 1\leq j\leq k$ and hence  $\Delta_E \big(p\circ \Phi(x) \big)= 0$, as $\Delta_E =\Delta_{E_1}+\dots +\Delta_{E_k}$.

\section{Holomorphic representations}

Let $J$ be a simple Euclidean Jordan algebra, denote by $\mathbb J$ its complexification. Form the corresponding \emph{tube-type domain} $T_\Omega = J+i\Omega\subset \mathbb J$. Let $G(T_\Omega)$ be the group of bi-holomorphic automorphisms of $T_\Omega$, which turns out to be a Lie group. Its neutral component $G(T_\Omega)^0$ is generated by
\smallskip

$\bullet$ the translations $t_u : z\mapsto z+u, u\in J$
\smallskip

$\bullet$ the group $L=\Str(J)^0=G(\Omega)^0$
\smallskip

$\bullet$ the inversion $\iota : z\mapsto -z^{-1}$\ 
\smallskip

 (see \cite{fk} Ch. X).
 
 It turns out to be wise to work with a two-fold covering of the group $G=G(T_\Omega)^0$. Notice that already for the upper-half plane $\mathcal H=\mathbb R+i\mathbb R^+$ in $\mathbb C$, the relevant group is $SL(2,\mathbb R)$, whereas the group of holomorphic diffeomorphisms of $\mathcal H$ is $PSL(2,\mathbb R) = SL(2,\mathbb R)/\{\pm \id\}$.

Viewing $\mathbb J$ as a complex Jordan algebra, there is a corresponding structure group $\Str(\mathbb J)$ which is a complexification of $\Str(J)$. Let $\mathbb L$ be its neutral connected component, which can be called \emph{the} complexification of $L$. The character $\chi$, being defined by an algebraic condition \eqref{defchi}, has a natural complex extension to $\mathbb L$. 

We now recall the construction of a two-fold covering group of $G(T_\Omega)^0$.
For $g\in G(T_\Omega)^0$ and $z\in T_\Omega$ denote by $D(g,z)$ the differential of $g$ at $z$. A remarkable fact is that $Dg(z)$ belongs to $\mathbb L$. This is obtained by verifying the property for the generators of $G$ and extending to the full group by the chain rule.

Let $\chi(g,z) = \chi\big(Dg(z)\big)$. This is a smooth cocycle, and the covering is defined using a square root of this cocycle. Let
\begin{equation}
\widetilde G = \{ (g,\psi_g), g\in G(T_\Omega)^0,\psi_g : T_\Omega\rightarrow \mathbb C \text{ holomorphic},\psi_g(z)^2 = \chi(g,z)\}, 
\end{equation}
with the group law
\[(g,\psi_g)(h,\psi_h) = (gh,(\psi_g\circ h) \,\psi_h)\ .
\]
Then $\widetilde G$ has a natural structure of Lie group, and the projection $(g,\psi_g)\longmapsto g$ is a twofold covering of $G$. Elements of $\widetilde G$ will denoted simply by $g$ and we let $\psi(g,z) = \psi_g(z)$ be the corresponding choice of a square root of $\chi(g,z)$.

The group $\widetilde G$ has generators very similar to those for $G$.

$\bullet$ For $u\in $, $\chi(t_u,z) \equiv 1$, so that the element $(t_u,1)$ belongs to $\widetilde G$ and is (with some abuse of notation) still denoted by $t_u$.
\smallskip

$\bullet$ For the inversion $\iota$, $\chi(\iota,z) = (\det z)^{-2}$, so $(\iota,{(\det z)}^{-1})$ belongs to $\widetilde G$ and is (again with some abuse of notation) denoted by $\iota$.
\smallskip

$\bullet$ Finally, let $\widetilde  L = \{ (\ell, \psi_{\ell} = \pm \chi(\ell)^{1/2}), \ell\in L=\Str(J)^0\}$, a twofold covering of $L$. 
\smallskip

The group $\widetilde G$ is generated by the translations $\{t_u,u\in J\}$, the group $\widetilde  L$ and the inversion $\iota$.

It turns out to be necessary to consider also the twofold covering $\widetilde {\mathbb L}$ of $\mathbb L=\Str(\mathbb J)^0$, constructed exactly as the twofold covering of $L$, using a square root of the complex-valued character $\chi$ of $\Str(\mathbb J)$. Denote by $ \psi $ the  corresponding (well-defined) character on $\widetilde {\mathbb L}$. It is worthwile to notice that for any $g\in \widetilde G$, the function $z\longmapsto \psi(g,z)^{-1}$ is (the restriction to $T_\Omega$ of) a polynomial function on $\mathbb J$.
 
Let $m$ be an integer. For any holomorphic function $F$ on $T_\Omega$ and any element $g\in \widetilde G$, let
\[\pi_m(g)F(z) = \psi(g^{-1},z)^m F\big(g^{-1}(z)\big)\ .
\]
This defines a smooth representation of $\widetilde G$ on the space $\mathcal O(T_\Omega)$ of holomorphic functions on $T_\Omega$, equipped with the Montel topology. 

Let us write the expression of the representation $\pi_m$ for generators of the group $\widetilde G$.

For translations $t_v, v\in J$,
\begin{equation}
\big(\pi_m(t_v) F\big)(z) = F(z-v)\ ,
\end{equation}
for an element $\ell\in \widetilde L$
\begin{equation}
\big(\pi_m(\ell)F\big)(z) = \psi(\ell)^{-m} F(\ell^{-1} z)\ ,
\end{equation}
for the inversion $\iota$,
\begin{equation}
\big(\pi_m(\iota)F\big)(z) = (\det z)^{-m} F(-z^{-1})\ .
\end{equation}
This construction of a twofold covering can be done in the context of the algebra $J(\mathbf c)$ for $\mathbf c$ a complete system of mutually orthogonal idempotents of $J$. We skip details, but just mention that all the results proved in Section 1 (and particularly Theorem \ref{surjkerrest}) have an almost trivial extension to the groups $\widetilde L, \widetilde G, \widetilde M_{\mathbf c}$ and $\widetilde L_j,\widetilde G_j$.

Similarly, for each $j, 1\leq j\leq k$, we may define representations of the group $\widetilde G_j$ on the space of holomorphic $\mathcal O(T_{\Omega_j})$ by

\[\pi^{(j)}_{m_j}(g) F(z_j) = \psi_j(g^{-1},z_j)^{m_j}F\big(g^{-1}(z_j)\big)\ .
\]
Given $\mathbf m = (m_1,m_2,\dots, m_k)$, consider the tensor product representation $\pi^{(1)}_{m_1}\otimes \pi^{(2)}_{m_2}\otimes\dots \otimes \pi^{(k)}_{m_k}$ of $\widetilde G_1\times \dots \times \widetilde G_k$ defined on $\mathcal O(T_{\Omega_{\mathbf c}})$ by
\[\pi^{(1)}_{m_1}(g_1)\otimes \pi_{m_2}(g_2)\otimes\dots \otimes \pi_{m_k}(g_k) F(z_1,z_2,\dots z_k) 
\]
\[=\prod_{j=1}^k\psi_j(g_j^{-1},z_j)^{m_j}\ F\big(g_1^{-1}(z_1),g_2^{-1}(z_2),\dots,g_k^ {-1}(z_k)\big)\ .
\]

\begin{proposition} The restriction map $\res : F\longmapsto F_{\vert T_{\Omega_{\mathbf c}}}$ intertwines the restriction of $\pi_m$ to $\widetilde G(\mathbf c)$ and the tensor representation $\pi_m^{(1)}\otimes \pi_m^{(2)}\otimes \dots\pi_m^{(k)}$ of $\widetilde G_1\times \dots \times \widetilde G_k$.
\end{proposition}
\begin{proof} First notice that for $g\in \widetilde G(\mathbf c)$ and $z\in T_{\Omega_{\mathbf c}}$, $\pi_m(g) F(z)$ only depends on   $g\!\!\!\mod \widetilde  M_\mathbf c$. Next use Proposition \ref{chichij} (more exactly its analog for the characters $\psi$ and $\psi_j$) and the conclusion follows.
\end{proof}

\section{The main theorem}

In this section, we want to construct a differential operator $D$ on $T_\Omega$  such that $res(\mathbf c)\circ D$ intertwines  the restriction of $\pi_m$ to $\widetilde  G(\mathbf c)$ and $\bigotimes_{j=1}^k \pi^{(j)}_{m_j}$ for an appropriate choice of $\mathbf m =(m_1,m_2,\dots,m_k)$.

\subsection{The data and the statement of the main theorem}

\begin{theorem}\label{maintheorem}
Let $(E,\Phi)$ be a regular representation of $J$ and assume that $N=\dim E = 2rm$ for some $m\in \mathbb N$.

 Let $q$ be a polynomial on $J$ which satisfies 
\begin{equation}\label{ass}
\begin{split}
&i)\quad q  \text { is } {\mathbf c}\text{-homogeneous of multidegree }  \mathbf p= (p_1,p_2,\dots, p_k)\\
&ii) \quad p=q\circ Q \ is\  {\mathbf c}\!\!-\!\!\text{pluriharmonic on } E\ .
\end{split}
\end{equation}
 Let 
\[m_1=m+2p_1,\quad  m_2=m+2p_2,\quad \dots,\quad  m_k = m+2p_k\ .\]
Let $D_q$ be the holomorphic constant coefficients differential operator on $\mathbb J$ with algebraic symbol $q$. Then $D_q$ satisfies,
for any $g\in \widetilde G(\mathbf c)$  whose restriction to  $J(\mathbf c)$ is equal to $(g_1,g_2,\dots, g_k)$
\begin{equation}\label{maing}
\res\circ D_q\circ\pi_m(g) = (\pi^{(1)}_{m_1}(g_1)\otimes \dots\otimes \pi^{(k)}_{m_k}(g_k)) \circ \res\circ D_q\ .
\end{equation}
\end{theorem}

\subsection{ The proof of the main theorem}

It suffices to verify \eqref{maing} for generators of $\widetilde G(\mathbf c)$. First, the operator $D$ has constant coefficients, hence commutes to translations and \eqref{maing} follows for $g=t_v, v\in J(\mathbf c)$.

When $g$ is in $\widetilde  L(\mathbf c)$, the proof to get \eqref{maing} is longer and we state it as a lemma. Notice however that the intertwining property \eqref{maing} is trivially satsified for $m\in \widetilde M(\mathbf c)$.

\begin{lemma}\label{mainL}
 For $\ell\in \widetilde L(\mathbf c)$ and for $z\in T_{\Omega(\mathbf c)}$,
\begin{equation}\label{mainl}
D\circ\pi_m(\ell)F (z) = (\pi^{(1)}_{m_1}(\ell_1)\otimes \dots\otimes \pi^{(k)}_{m_k}(\ell_k)) (DF_{\vert T_{\Omega(\mathbf c)}})(z)\ .
\end{equation}
\end{lemma}
\begin{proof} As $D$ is a constant coefficients differential operator, it suffices to prove \eqref{mainl} for the family of functions $f_v,v\in \mathbb J$, where $f_v(z) = e^{(z,v)}$. 

Let $\ell\in \widetilde L(\mathbf c)$. For any $z\in T_\Omega$,

\[\pi_m(\ell) f_v (z) = \psi(\ell)^{-m} f_{{\ell^{-1}}^t v}(z)\ .
\]
As $Df_v =  q(v) f_v$,
\begin{equation*}
D\circ \pi_m(\ell) f_v (z)= \psi(\ell)^{-m}  q({\ell^{-1}}^t v)   f_{{\ell^{-1}}^tu}(z)\ ,
\end{equation*}
so that, using \eqref{covLq}
\begin{equation}\label{mainleft}
D\circ \pi_m(\ell) f_v (z)= \prod_{j=1}^k \psi_j(\ell_j)^{-m-2p_j} q(v) f_{{\ell^{-1}}^tv}(z)\ .
\end{equation}
Let further $\proj(\mathbf c) v = v_1+v_2+\dots v_k$, where $v_j\in J_j$.
For  $z=z_1+z_2+\dots+z_k \in T_{\Omega(\mathbf c)}$,
\begin{equation*} f_v(z) = \prod_{j=1}^k e^{i(z_j,v_j)}\ ,
\end{equation*}
and  hence
\[DF_{\vert T_{\Omega(\mathbf c)}} (z)=  q(v) \prod_{j=1}^k e^{i(z_j,v_j)}\ , \]
so that
\begin{equation}\label{mainright}
(\pi_{m_1}^{(1)}(\ell_1)\otimes \dots\otimes \pi_{m_k}^{(k)}(\ell_k)) (DF_{\vert T_{\Omega(\mathbf c)}})(z)
=  q(v) \prod_{j=1}^k \psi_j(\ell_j)^{-m_j} \prod e^{(z_j, {\ell_j^{-1}}^t u_j)}\ .
\end{equation}
Now compare \eqref{mainleft} and \eqref{mainright} to finish the proof of Lemma \ref{mainL}.
\end{proof}
It remains to prove \eqref{maing} for $g=\iota$, which we also state as a lemma.
\begin{lemma} Let $z\in T_{\Omega(\mathbf c)}$. Then
\begin{equation}\label{mainiota}
D\circ\pi_m(\iota)F (z) = (\pi_{m_1}^{(1)}(\iota_1)\otimes \dots\otimes \pi_{m_k}^{(k)}(\iota_k)) (DF_{\vert T_{\Omega(\mathbf c)}})(z)\ .
\end{equation}

\end{lemma}
\begin{proof} For $\xi\in E$, let
\[F_\xi(z) = e^{\frac{i}{2}\big(z,Q(\xi)\big)}\ .
\]

As the representation $(E,\Phi)$ is assumed to be regular, it is enough to prove \eqref{mainiota} for the family of functions $F_\xi,\xi\in E$.

For $z\in T_\Omega$
\[\pi_m(\iota) F_\xi (z) =(\det z)^{-m} e^{\frac{i}{2}\big(-z^{-1},Q(\xi)\big)}
\]
Notice that for $z=x+iy \in T_\Omega$ the quadratic form \[\big(y,Q(\xi)\big)=\langle \Phi(y^{\frac{1}{2}}) \xi, \Phi(y^{\frac{1}{2}} )\xi\rangle\] is positive-definite, and hence,  $\xi\longmapsto F_\xi(z)$ is in the Schwartz class $\mathcal S(E)$. Use \eqref{Fourier} for $p\equiv1$ on $E$ and notice that $m=\frac{N}{2r}$ to get
\[\pi_m(\iota) F_\xi (z)  = (2\pi)^{-\frac{N}{2}} (-i)^{rm} \int_E e^{i\langle \eta,\xi\rangle} e^{\frac{i}{2}(z,Q(\eta))} d\eta\ .
\]
Now use \eqref{construcD} to get
\[D\pi_m(\iota) F_\xi (z) =  (2\pi)^{-\frac{N}{2}}(- i)^{rm}2^{-\mathbf r.\mathbf p}\int_E q\big(Q(\eta)\big) e^{i\langle\xi,\eta\rangle} e^{\frac{i}{2}\big(z,Q(\eta)\big)} d\eta\ ,
\]
where $\mathbf r.\mathbf p = \sum_{j=1}^k r_jp_j$.

Next, assume that $z=z_1+z_2+\dots+z_k\in T_{\Omega(\mathbf c)}$. Then, as
$\proj_{\mathbf c}(Q(\eta)) = \sum_{j=1} Q_j(\eta_j)$, 
\[D\pi_m(\iota) F_\xi(z) =(2\pi)^{-\frac{N}{2}}(-i)^{rm}2^{-\mathbf r.\mathbf p}\]
\[\int_{E_1}\dots\int_{E_k} p(\eta_1,\dots,\eta_k)\prod_{j=1}^k e^{i\langle\xi_j,\eta_j\rangle}e^{\frac{i}{2}\big(z_j,Q_j(\eta_j)\big)}d\eta_1\dots d\eta_k\ .
\]
Use Fubini theorem and \eqref{Fourier} repeatedly  for $j=1,\dots, k$, take into account that  $N_j=2 r_jm$ (see \eqref{NrNjrj}) to get
\[D\pi_m(\iota) F_\xi(z) =2^{-\mathbf r.\mathbf p} \prod_{j=1}^k \det (\frac{z_j}{i})^{-m}
e^{\frac{i}{2} (-z_j^{-1},Q_j(\xi_j)}
p(\dots, \Phi(-z_j^{-1})\xi_j,\dots)
\]
Now use \eqref{NrNjrj} and \eqref{pchom} to obtain
\begin{equation}\label{leftiota}
D\pi_m(\iota) F_\xi(z) = 2^{-\mathbf r.\mathbf p}\prod_{j=1}^k \det (\frac{z_j}{i})^{-m-2p_j}e^{\frac{i}{2} \big(-z_j^{-1},Q_j(\xi_j)\big)}p(\xi_1,\dots,\xi_k)\ .
\end{equation}
On the other hand, for $z\in T_\Omega$,
\[DF_\xi(z) =  q(\frac{1}{2}Q(\xi))\,e^{\frac{i}{2}\big(z,Q(\xi)\big)}=2^{-\mathbf r.\mathbf p} p(\xi)e^{\frac{i}{2}\big(z,Q(\xi)\big)}\ .
\]
Hence for  $z=z_1+z_2\dots+z_k \in T_{\Omega(c)}$ 
\[DF_\xi(z) =2^{-\mathbf r.\mathbf p} p(\xi) \prod_{j=1}^k e^{\frac{i}{2} \big(z_j,Q(\xi_j)\big)}
\]
so that
\begin{equation}\label{rightiota}
\begin{split}
\big(\pi^{(1)}_{m_1}(\iota_1)\otimes \dots\otimes \pi^{(k)}_{m_k}(\iota_k)\big)D{F_\xi}_{\vert T_{\Omega(\mathbf c)}}(z_1,z_2,\dots,z_k)\\ 
= 2^{-\mathbf r.\mathbf p} p(\xi) \prod_{j=1}^k {\det}_j(\frac{z_j}{i})^{-m_j} \,e^{\frac{i}{2} \big(z_j,Q(\xi_j)\big)}\ .
\end{split}
\end{equation}
Compare \eqref{leftiota} and \eqref{rightiota} to conclude.
\end{proof}
This achieves the proof of the main theorem.

\section{Examples in rank $2$}

In this section, let $J=J_n$ be the simple Euclidean Jordan algebra of  rank $2$  and dimension $n$, i.e. $J_n=\mathbb R \oplus \mathbb R^{n-1}$ , with the Jordan product
\[(s,x_1,\dots, x_{n-1})(t, y_1,\dots, y_{n-1}) = (st+x.y,sy_1+tx_1, \dots, sy_{n-1}+tx_{n-1}) \ ,
\]
where $x.y = x_1y_1+\dots+x_{n-1}y_{n-1}$.

We will assume that $n\geq 4$. In fact, for $n=2$,  the algebra $J_2$ is not simple. For $n=3$, $J$ is isomorphic to $Symm(2,\mathbb R)$ and so this case has been studied in \cite{i} and differs slightly from the general case (see next footnote).

The trace and the determinant are given by
\[\tr(s,x) = 2s,\qquad \det(s,x) = s^2-\vert x\vert^2\ .\]

The structure group of $J$ is the product $\mathbb R^*\times O(1,n-1)$ and the group denoted by $\widetilde L$ in the general case is equal to $\mathbb R^+\times Spin_0(1,n-1)$.

The cone $\Omega = \Omega_n$ is the \emph{forward cone}
\[\Omega = \{ (s,x)\in \mathbb R\times \mathbb R^{n-1},\quad  s^2-\vert x\vert^2 >0,\ s>0\}\ .
\]
The tube-type domain is 
\[T_\Omega= \{Z=(z_0,z_1,\dots, z_{n-1})\in \mathbb C^n,\  \im(Z) \in \Omega\}\ ,
\]
the group of biholomorphic automorphisms is isomorphic to $O(2,n)/\{ \pm \id\}$, and  the group $\widetilde G$ is isomorphic to $Spin_0(2,n)$, see \cite{s}. The domain has a bounded realization known as the \emph{Lie ball}.

It will be convenient to use for $J$ the standard inner product on a Euclidean Jordan algebra, which in this case is given by 
\[\big( (s,x) , (t,y) \big)= 2 (st+x.y)\ .
\]
An idempotent of rank $1$ is of the form $(\frac{1}{2}, x)$, with $\vert x\vert = \frac{1}{2}$. Up to an isomorphism of $J$, there is only one (non-trivial) CSOI, namely
\[\mathbf c = (c_1,c_2) ,\qquad  c_1=\big(\frac{1}{2}, \frac{1}{2},0,\dots,0\big),\quad c_2 = \big(\frac{1}{2}, -\frac{1}{2},0,\dots,0\big).
\]
Consequently,
\[J(\mathbf c) =\mathbb R c_1\oplus \mathbb Rc_2.\]
The corresponding  Peirce  decomposition is given by
\[J = \mathbb R c_1\oplus \mathbb R c_2 \oplus J_{\frac{1}{2}}
\]
where 
\[J_{\frac{1}{2}}=\{(s,x), s=0, x_1=0\}\simeq \mathbb R^{n-2}\ .\]

A well-adapted orthonormal basis of $J$ is given by
\[f_1 =c_1,\quad  f_2=c_2,\quad  \text{and for } j\geq 3\quad  f_j = (0,\dots, 0,\frac{1}{\sqrt 2},0,\dots, 0)
\] 
where $\frac{1}{\sqrt 2}$ is in the $j$-th place. A generic element of $J$  will be denoted by \[y=(y_1,y_2,\dots, y_j,\dots,y_n) =(y_1,y_2,y')
\]
where $y_j$ is the $j$-th coordinate of the element  in the new basis $\{ f_1,\dots, f_n\}$. The formula for the base change  is given by
\[y_1 = s+x_1,\quad y_2 = s-x_1,\quad y_j = \sqrt{2} \,x_{j-1}, \ 3\leq j\leq n\ .
\]

The group $L(\mathbf c)$ preserves $J(\mathbf c)$ and hence also its orthogonal $J_{\frac{1}{2}}$. In the basis $\{ f_1,\dots, f_n\}$, an element of $L(\mathbf c)$ is represented by 
\[\ell(u,v,m) = \left\{\begin{pmatrix}\begin{matrix} ue^v& 0 \\ 0&ue^{-v} \end{matrix} &\quad 0&\ \\ \\0&\quad u\,m\quad &\\ & & \end{pmatrix},\quad \left\{\begin{matrix} u\in \mathbb R_+,\  v\in \mathbb R\\  m\in Spin(n-2)\end{matrix}\right.\ \right\}.
\]
  
$L_1$ and $L_2$ are both isomorphic to $\mathbb R_+$, and the restriction map is given by
\[L \ni\ell(u, v, m)\quad \longmapsto \quad  (ue^v, ue^{-v})\in L_1\times L_2\ .
\]
The group $\widetilde M(\mathbf c)$ is isomorphic to $Spin(n-2)$.

\begin{lemma}
A polynomial $q$ on $J$ is $\mathbf c$-homogeneous of multi-degree $(p_1,p_2)$ if and only if $q$ is of the form
\begin{equation}
q(y) = \sum_{j=0}^{\inf(p_1,p_2)} a_j\, y_1^{p_1-j} y_2^{p_2-j} \vert y'\vert^{2j}
\end{equation}
for some  $a_j\in \mathbb C$.
\end{lemma}
\begin{proof} The group $M(\mathbf c)$ is isomorphic to $Spin(n-2)$, fixes the subalgerba $J(\mathbf c)$ and acts as $SO(n-2)$ on $J_{\frac{1}{2}}$. Hence\footnote{Here is the reason to exclude the case $n=3$, as $SO(1)=\{\id\}$ and the invariance by $M(\mathbf c)$ imposes no condition in this case.}
 a $\mathbf c$-homogeneous polynomial can be written as a polynomial in $y_1,y_2$ and $\vert y'\vert^2$.
 
 Now consider the elementary polynomial $y_1^{m_1} y_2^{m_2} \vert y'\vert^{2m_3}$. It satisfies the required conditions for $\mathbf c$-homogeneity if and only if $m_1,m_2,m_3$ satisfy
 \begin{equation}
  \begin{split}
  m_1+m_2+m_3&= p_1+p_2\\
  m_1-m_2 &= p_1-p_2\ .
  \end{split}
  \end{equation}
 Hence $m_1=p_1-j, m_2=p_2-j, m_3 = j$ for some $j, 0\leq j\leq p_1, p_2$ and the conclusion follows.
\end{proof}
The cone $\Omega(c)$ is equal to $\{ a_1 c_1 +a_2 c_2, a_1, a_2\in \mathbb R\} $ is a product of two positive half-lines, the tube-domain $T_{\Omega(\mathbf c)}$ is equal to \[\{(z_0,z_1,0,\dots,0), \im(z_0+z_1)>0, \im(z_0-z_1)>0\}\] and is a product of two complex upper half-planes, the groups $\widetilde G_1$ and $\widetilde G_2$ are  isomorphic to $Sl(2,\mathbb R)$.

Let $(E, \Phi)$ be a representation of $J$. For $v, w\in \mathbb R^{n-1}$
\[(0,v)(0,w) = (v.w, 0,\dots,0)\ ,\qquad v.w=\sum_{j=1}^{n-1}v_jw_j\ ,
\]
so that 
\[ \Phi\big((0,v)\big) \Phi\big((0,w)\big)+\ \Phi\big((0,w)\big) \Phi\big((0,v)\big)= 2\,v.w \Id_E\ .
\]
Hence $E$ is a \emph{Clifford module} for the Clifford algebra $Cl(n-1)$ generated by $\mathbb R^{n-1}$ with the relation (beware of the absence of sign $-$)
\[vw+wv = 2\, v.w\ .
\]
Conversely, if $E$ is a Clifford module for $Cl(n-1)$, then set
\[(s,x)\in J,\,\xi\in E,\qquad \Phi\big((s,x)) \xi = s\xi+x\xi
\]
to obtain a representation of $J$. For a deeper study of these representations, see \cite{c92}.

Let $E=E_1\oplus E_2$  the decomposition of $E$ with respect to the CSOI $\mathbf c=(c_1,c_2)$. For $v\in  \mathbb R^{n-2}$,\quad 
$c_1(0,0,v)=\frac{1}{2}(0,0,v)$ and hence $\Phi\big((0,0,v)\big) = 2\Phi\big(c_1(0,0,v)\big) = \Phi(c_1)\Phi\big((0,0,v)\big) +\Phi\big((0,0,v)\big)\Phi(c_1) $, so that 
\[\Phi\big((0,0,v)\big) \Phi(c_2) = \Phi(c_1)\Phi\big((0,0,v)\big)\ .
\]
Hence $\Phi\big((0,0,v)\big)$ permutes $E_1$ and $E_2$. 

The quadratic map $Q$ is given in the original basis of $J$ by
\[Q(\xi_1+\xi_2) = \big(\frac{1}{2}(\vert \xi_1\vert^2)+\vert \xi_2\vert^2),\frac{1}{2}(\vert \xi_1\vert^2-\vert \xi_2\vert^2),\dots,\langle\Phi(e_j)\xi_1,\xi_2\rangle,\dots \big)\ .
\]
and hence in the basis $\{ f_1,f_2,\dots, f_n\}$ by
\[Q(\xi_1+\xi_2) = \big( \vert \xi_1\vert^2, \vert \xi_2\vert^2, \dots, 2\langle\Phi(f_j)\xi_1,\xi_2\rangle, \dots\big)\ .
\]
 Let $\{\epsilon_k, 1\leq k\leq N_1\}$ be an orthogonal basis of $E_1$ and denote by $(\xi_{1,k})_{1\leq k\leq N_1}$ the corresponding coordinates of a generic element $\xi_1\in E_1$. Finally, denote by $\Delta_1$ the partial Laplacian on $E$ related to $E_1$, i.e. 
 \[\Delta_1 = \sum_{k=1}^{N_1}  \frac{\partial^2}{\partial \xi_{1,k}^2}\ .
 \]

For the next statements, let $\displaystyle \partial _j=\frac{\partial}{\partial y_j}$ be the partial derivative (on $J$) with respect to the coordinate $y_j$. 
\begin{proposition}Let $q$ be a polynomial on $J$ and let $p$ be the polynomial on $E$ defined by $p= q\circ Q$. Then for $\xi=(\xi_1,\xi_2)$
\[\Delta_1 p (\xi_1,\xi_2) = (\delta_1q)\big( 
Q(\xi_1,\xi_2)\big)
\]
where 
\begin{equation}\delta_1= 2 N_1 \partial_1+ 4\, y_1\, \partial_1^2 + 4 \sum_{j=3}^n y_j\, \partial_1 \partial_j+2 y_2\, \sum_{j=3}^n \partial_j^2\ .
\end{equation}
\end{proposition}
\begin{proof}
Consider $p$ as a polynomial on $E_1\oplus E_2$, so that
\[p(\xi_1,\xi_2) = q\big( \vert \xi_1\vert^2,  \vert \xi_2\vert^2,\dots,2 \langle \xi_1,\Phi(f_j)\xi_2\rangle,\dots\big)
\]
Then, for any $k, 1\leq k\leq N_1$
\[\frac{\partial}{\partial \xi_{1,k}}\, p\,(\xi_1,\xi_2) = 2\, \xi_{1,k} \,\partial_1q+ 2 \sum_{j=3}^n \langle \epsilon_k,\Phi(f_j)\xi_2\rangle\,\partial_jq
\]
\[\frac{\partial^2}{\partial \xi_{1,k}^2}\, p\,(\xi_1,\xi_2) = 2\,\partial_1 q +4\, \xi_{1,k}^2\, \partial_1^2 q+ 8 \,\xi_{1,k} \sum_{j=3}^n \langle \epsilon_k,\Phi(f_j)\xi_2\rangle\, \partial_1\partial_j q
\]
\[+ 4
\sum_{j=3}^n \sum_{i=3}^n \langle \epsilon_k,\Phi(f_j)\xi_2\rangle\ \langle \epsilon_k,\Phi(f_i)\xi_2\rangle\, \partial_i\partial_j q
 \ .
 \]
 Sum over $k, 1\leq k\leq N_1$ and use the formula
 \[\sum_{k=1}^{N_1} \langle \epsilon_k,\Phi(f_j)\xi_2\rangle\langle \epsilon_k,\Phi(f_i)\xi_2\rangle= \langle \Phi(f_j) \xi_2, \Phi(f_i\xi_2)\rangle\ ,
 \]
 to get 
\[\Delta_1 p(\xi_1,\xi_2) = 2\, N_1 \partial_1 q+ 4\vert \xi_1\vert^2 \partial_1^2 q + 8 \sum_{j=3}^n \langle \xi_1,\Phi(f_j) \xi_2\rangle\partial_1 \partial_j q
\]
\[ +4\,\sum_{i=3}^n\sum_{j=3}^n \langle \Phi(f_j)\xi_2,\Phi(f_i)\xi_2 \rangle \,\partial_i\partial_j \,q\ .
\] 
For $3\leq i\neq j\leq n$, $f_i f_j = 0$ , so that $\Phi(f_i) \Phi(f_j) = -\Phi(f_j) \Phi(f_i)$, whereas $f_j^2 = \frac{1}{2} e$ and hence $\langle \Phi( f_j) \xi_2, \Phi(f_j) \xi_2\rangle = \frac{1}{2}\vert \xi_2\vert^2$. The formula follows from these observations.
\end{proof}

\begin{lemma} Let $p_1,p_2\in \mathbb N$  and let $k\in \mathbb N,\ 0\leq k\leq \inf(p_1,p_2)$.
Then
\begin{equation}\label{deltamonomial}
\begin{split}
&\delta_1 \big(y_1^{p_1-k} y_2^{p_2-k} \vert y'\vert^{2k}\big) =\\ &(p_1-k) \Big(4k+2\,N_1+4 p_1-4 \Big) y_1^{p_1-k-1} y_2^{p_2-k} \vert y'\vert^{2k}\\ & +2k(2k+n-3) y_1^{p_1-k} y_2^{p_2-k+1} \vert y'\vert^{2k-2}\ .
\end{split}
\end{equation}
\end{lemma}

The verification is elementary and left to the reader.

\begin{proposition}\label{recq}
 Let $q\not\equiv 0$ be the polynomial defined by
\begin{equation}
q(y) = \sum_{j=0}^{\inf(p_1,p_2)} a_j\, y_1^{p_1-j} y_2^{p_2-j} \vert y'\vert^{2j}\ .
\end{equation}
Then $p(\xi) = q(Q(\xi))$ is $\mathbf c$-pluri-harmonic if and only if $ p_1=p_2=p$ and the coefficients $a_j$ satisfy the relation
\begin{equation}\label{recaj}
(j+1)\big(j+\frac{n-1}{2}\big) a_{j+1} + (p-j)(j+ \frac{N_1}{2}+p-1) a_j=0\ .
\end{equation}
\end{proposition}
\begin{proof}
The polynomial $p=q\circ Q$ is $\mathbf c$-pluriharmonic if and only if
\begin{equation}\label{delta=0}
\delta_1 q=0,\qquad \delta_2 q = 0\ ,
\end{equation}
where 
\[\delta_2 = 2N_1\partial_2+4y_2 \partial_2^2+4 \sum_{j=3}y_j \partial_2\partial_j+2 y_1 \sum_{j=3}^n \partial_j^2\ .
\]
Use \eqref{deltamonomial} to compute $\delta_1 q$ and simlarly for $\delta_2q$. The conditions  \eqref{delta=0} are satisfied if and only if, for any $k, 0\leq k\leq \inf(p_1,p_2)-1$
\begin{equation}
\begin{split}
(k+1)(k+\frac{n-1}{2}) a_{k+1}+(p_1-k)\big(k+\frac{N_1}{2}+p_1-1\big)a_k &= 0\\
(k+1)(k+\frac{n-1}{2}) a_{k+1}+(p_2-1)\big(k+\frac{N_1}{2}+p_2-1\big)a_k &= 0\ .
\end{split}
\end{equation}
Hence $p_1=p_2$ and \eqref{recaj} follows.
\end{proof}

From Proposition \ref{recq} follows
\begin{equation}
a_j = -\,\, \frac{(p+1-j)\big(j+\frac{N_1}{2}+p-2\big)}{j\,\big(j+\frac{n-3}{2}\big)} a_{j-1}
\end{equation}
and hence
\begin{equation}
a_j = (-1)^j  \frac{p\dots (p-(j-1))\ (\frac{N_1}{2}+p-1)\dots  (\frac{N_1}{2}+p-1+(j-1)) }{1\, 2\dots j \ (\frac{n-1}{2})\dots \big(\frac{n-1}{2}+(j-1)\big)}\ a_0
\end{equation}
For $n,m,p\in \mathbb N$, let for $j\in \mathbb N,\ 0\leq j\leq p$

\[a_j^{n,m,p}= (-1)^j \frac{p!}{j!(p-j)!}\,\frac{(m+p-1)\dots \big(m+p-1+(j-1)\big)}{(\frac{n-1}{2})\dots (\frac{n-1}{2}+j-1)}\ .
\]
\begin{theorem}\label{theoremrank2}
 Let $E$ be a Clifford module for $Cl(n-1)$ and assume that the corresponding representation of $J$ is regular. Assume moreover that $\dim E = 4m$ for some $m\in \mathbb N$.

Let $D$ be the holomorphic diffferential operator defined on $T_\Omega$ by 
\begin{equation}
D= \sum_{j=0}^p a_j^{n,m,p} \left(\frac{\partial^2}{\partial z_0^2}-\frac{\partial^2}{\partial z_1^2}\right)^{p-j} \Delta_{n-2}^j\ .
\end{equation}
For any $g\in \widetilde G(\mathbf c)$, whose restriction to $T_{\Omega(\mathbf c)}$ is equal to $(g_1,g_2)$, the operator $D$ satisfies
\begin{equation}
(\res \circ D)\circ \pi_{m} (g) =\big( \pi_{m+2p}(g_1) \otimes \pi_{m+2p} (g_2)\big) \circ (\res\circ D)\ .
\end{equation}
\end{theorem}
\begin{proof} Recall that $N_1=\dim E_1=\frac{1}{2} \dim E = 2m$, and hence $\frac{N_1}{2}=m$. Theorem \ref{theoremrank2} is merely a transcription of Theorem \ref{maintheorem} in the present context.
\end{proof}
The Clifford modules for $Cl(n-1)$ viewed as representations of the algebra $J$ are studied in \cite{c92}, and in particular  the regular representations  are classified (Th\'eor\`eme 3), thus giving for each $n$ the  values of $m$ for which Theorem \eqref{theoremrank2} is valid.

\footnotesize{\noindent Address\\ Jean-Louis Clerc, Universit\'e de Lorraine, CNRS, IECL, F-54000 Nancy, France
\medskip

\noindent \texttt{{jean-louis.clerc@univ-lorraine.fr
}}

\end{document}